\documentclass{amsart}
\usepackage{amssymb,latexsym}
\usepackage{amscd,amsthm}

\newtheorem{theorem}{Theorem}[section]
\newtheorem{lemma}[theorem]{Lemma}

\newtheorem{proposition}[theorem]{Proposition}
\newtheorem{corollary}[theorem]{Corollary}

\theoremstyle{definition}
\newtheorem{definition}[theorem]{Definition}

\newtheorem*{remark}{Remark}

\DeclareMathOperator{\Ext}{Ext}
\DeclareMathOperator{\Hom}{Hom}
\DeclareMathOperator{\Tor}{Tor}

\DeclareMathOperator{\cok}{cok}

\newcommand{\cat}[1]{\mathcal{#1}}           

\newcommand{\class}[1]{\mathcal{#1}}   
\newcommand{\N}{\mathbb{N}}
\newcommand{\Z}{\mathbb{Z}}
\newcommand{\Q}{\mathbb{Q/Z}}
\newcommand{\mathcolon}{\colon\,} 

\newcommand{\ch}{\textnormal{Ch}(R)}
\newcommand{\cha}[1]{\textnormal{Ch}(\mathcal{#1})}
\newcommand{\rmod}{R\text{-Mod}}

\newcommand{\rightperp}[1]{#1^{\perp}}
\newcommand{\leftperp}[1]{{}^\perp #1}

\newcommand{\homcomplex}{\mathit{Hom}}

\begin{document}

\title{Gorenstein AC-projective complexes}

\author{James Gillespie}
\address{Ramapo College of New Jersey \\
         School of Theoretical and Applied Science \\
         505 Ramapo Valley Road \\
         Mahwah, NJ 07430}
\email[Jim Gillespie]{jgillesp@ramapo.edu}
\urladdr{http://pages.ramapo.edu/~jgillesp/}

\date{\today}

\begin{abstract}
Let $R$ be any ring with identity and $\ch$ the category of chain complexes of (left) $R$-modules. We show that the Gorenstein AC-projective chain complexes of~\cite{bravo-gillespie} are the cofibrant objects of an abelian model structure on $\ch$. The model structure is cofibrantly generated and is projective in the sense that the trivially cofibrant objects are the categorically projective chain complexes. We show that when $R$ is a Ding-Chen ring, that is, a two-sided coherent ring with finite self FP-injective dimension, then the model structure is finitely generated, and so its homotopy category is compactly generated. Constructing this model structure also shows that every chain complex over any ring has a Gorenstein AC-projective precover. These are precisely Gorenstein projective (in the usual sense) precovers whenever $R$ is either a Ding-Chen ring, or, a ring for which all level (left) $R$-modules have finite projective dimension. For a general (right) coherent ring $R$, the Gorenstein AC-projective complexes coincide with the Ding projective complexes of~\cite{Ding-Chen-complex-models} and so provide such precovers in this case.   
\end{abstract}

\maketitle

\section{introduction}\label{sec-intro}

With the goal of attaching a triangulated stable module category to a general ring, the Gorenstein AC-injective and Gorenstein AC-projective $R$-modules were introduced and studied in~\cite{bravo-gillespie-hovey}. It was shown there that the class of Gorenstein AC-projective modules form the cofibrant objects of an abelian model structure on the category $R$-Mod, of (left) $R$-modules. On the other hand, the dual Gorenstein AC-injectives are the fibrant objects of another model structure on $R$-Mod.  These concepts were extended to the category $\ch$, of chain complexes of $R$-modules, in~\cite{bravo-gillespie}. In particular, the Gorenstein AC-injective and Gorenstein AC-projective chain complexes were studied, and, the Gorenstein AC-injective complexes were shown to be the fibrant objects of a (cofibrantly generated) abelian model structure on $\ch$. However, as noted in the introduction to~\cite{bravo-gillespie}, the Gorenstein AC-projective model structure was not constructed there; it is much more technical to construct. It is the purpose of this paper to give this construction to complete the work in~\cite{bravo-gillespie}.   

Let us recall the definition of a Gorenstein AC-projective chain complex and give a precise statement of the main result in this paper.

\begin{definition}\label{def-Gorenstein AC-projective complex}
We call a chain complex $X$ \textbf{Gorenstein AC-projective} if there exists an exact complex of projective complexes $$\cdots \rightarrow P_1 \rightarrow P_0 \rightarrow P^0 \rightarrow P^1 \rightarrow \cdots$$ with $X = \ker{(P^0 \rightarrow P^1)}$ and which remains exact after applying $\Hom_{\ch}(-,L)$ for any level chain complex $L$; see Section~\ref{subsec-character duality} for the notion of a level chain complex.  
\end{definition}

It was shown in~\cite[Theorem~4.13]{bravo-gillespie} that $X$ is Gorenstein AC-projective if and only if each $X_n$ is a Gorenstein AC-projective $R$-module and any chain map $f : X \rightarrow L$ is null homotopic whenever $L$ is a level complex. For most rings commonly occurring in practice, the Gorenstein AC-projective complexes coincide with the usual Gorenstein projective complexes of~\cite{garcia-rozas}, or at least with the Ding projective complexes of~\cite[Section~3]{Ding-Chen-complex-models}. See the proof of Corollary~\ref{cor-precovers} and the following remarks at the end of Section~\ref{sec-Goren-AC-proj} for more precise statements.  

The theorem left open to prove is Theorem~\ref{them-Gorenstein AC-projectives complete cotorsion pair} below. We recall that by a \textbf{projective cotorsion pair} $(\class{W},\class{C})$ we mean a complete cotorsion pair, in some abelian category with enough projectives,  with $\class{W}$ thick (so closed under direct summands and satisfying the two-out-of-three property on short exact sequences) and such that $\class{W} \cap \class{C}$ is precisely the class of projective objects. By Hovey's correspondence between cotorsion pairs and abelian model structures~\cite{hovey}, such a cotorsion pair is equivalent to an abelian model structure on the category in which every object is fibrant, the objects in $\class{C}$ are cofibrant, and the objects in $\class{W}$ are trivial.  Such an abelian model structure is called $\textbf{projective}$ because the trivially cofibrant objects $\class{C} \cap \class{W}$ coincide with the projective objects. We now state the main result. 

\begin{theorem}\label{them-Gorenstein AC-projectives complete cotorsion pair}
Let $R$ be any ring and let $\class{GP}$ denote the class of Gorenstein AC-projective chain complexes. Set $\class{W} = \rightperp{\class{GP}}$, the right orthogonal with respect to $\Ext^1_{\ch}(-,-)$.
Then $(\class{GP}, \class{W})$ is a projective cotorsion pair in $\ch$. It is cogenerated by a set and so it is equivalent to a cofibrantly generated projective model structure on $\ch$. The homotopy category of this model structure is equivalent to $K(\class{GP})$, the triangulated category of all Gorenstein AC-projective chain complexes modulo the usual chain homotopy relation. 
\end{theorem}

We also point out that the homotopy category is a well generated category in the sense of ~\cite{neeman-well generated}. Indeed once we construct a cofibrantly generated model structure on a locally presentable (pointed) category, a main result from~\cite{rosicky-brown representability combinatorial model srucs} assures us that its homotopy category is well generated. So the point is to build a cofibrantly generated model structure, which due to the work of Hovey boils down to constructing a projective cotorsion pair that is \emph{cogenerated by a set}~\cite{hovey}.


Section~\ref{sec-Ding-Chen} concerns the case of when $R$ is a Ding-Chen ring in the sense of~\cite{ding and chen 93, ding and chen 96, gillespie-ding}. This is a two-sided coherent ring $R$ for which $R$ has finite self FP-injective (absolutely pure) dimension when viewed as either a left or a right module over itself. The result proved, Theorem~\ref{them-Ding-Chen case}, says a few things about the model structure of Theorem~\ref{them-Gorenstein AC-projectives complete cotorsion pair}. First, the identify functor from it to the Gorenstein AC-injective model structure of~\cite{bravo-gillespie} is a Quillen equivalence in this case. Second, the model structure is finitely generated and so it follows from a result of Hovey~\cite[Corollary~7.4.4]{hovey-model-categories} that the associated homotopy category is compactly generated. Finally,  Theorem~\ref{them-Ding-Chen case} gives a further description of the homotopy category.  In particular, we see that the homotopy category is equivalent to the chain homotopy category of all chain complexes $X$ (resp.  $Y$) with each component $X_n$ (resp. $Y_n$) a Gorenstein projective (resp. Gorenstein injective) $R$-module in the usual sense of~\cite{enochs-jenda-book}. This follows from the characterizations of Ding modules and complexes provided in~\cite{gillespie-ding-modules}.

The plan to prove Theorem~\ref{them-Gorenstein AC-projectives complete cotorsion pair} is to imitate the proof in~\cite{bravo-gillespie-hovey} of the Gorenstein AC-projective model structure on $R$-modules, which first built a Quillen equivalent model structure on chain complexes and then passed it down to the category of $R$-modules. We follow the same approach, working in $\textnormal{Ch}(\ch)$, the category of chain complexes of chain complexes. This is the same as the category of bicomplexes. However, changing signs to work with bicomplexes misses the point.  
Perhaps the correct perspective is to follow the idea in~\cite{gillespie-hovey-graded-gorenstein}. One can first identify $\ch$ with the category of graded $R[x]/(x^2)-$modules where $R[x]/(x^2)$ is thought of as a graded ring, with a copy of $R$ in degrees 0 and $-1$, and putting $x$ in degree $-1$. Then to imitate the proof in~\cite{bravo-gillespie-hovey} we should be working with chain complexes of graded $R[x]/(x^2)-$modules. However, for our purposes we find it be easier to just stick with the category $\textnormal{Ch}(\ch)$, and we refer to an object in this category as a \emph{double complex} or simply a \emph{complex of complexes}. The reason for this is mainly because the literature on chain complexes already has many handy references for the graded tensor product and Hom that we will use, and these are stated in terms of chain complexes and not graded $R[x]/(x^2)-$modules. So Section~\ref{sec-proj} shows how to construct projective model structures on double complexes. Then Section~\ref{sec-double-complexes} uses this to build a model structure on double complexes that is Quillen equivalent to the one in Theorem~\ref{them-Gorenstein AC-projectives complete cotorsion pair}. We finally are able to prove that main theorem in Section~\ref{sec-Goren-AC-proj} by passing the model structure on double complexes down to the ground category of chain complexes. We also point out at the end of Section~\ref{sec-Goren-AC-proj} how Theorem~\ref{them-Gorenstein AC-projectives complete cotorsion pair} provides for the existence of Gorenstein projective (or at least Ding projective) precovers in $\ch$ for the most commonly used coherent rings $R$. Section~\ref{sec-Ding-Chen} describes the special case when $R$ is a Ding-Chen ring.

\section{preliminaries}\label{sec-prelim}

Throughout the paper $R$ denotes a general ring with identity. An $R$-module will mean a left $R$-module, unless stated otherwise. The category of $R$-modules will be denoted $\rmod$ and the associated category of chain complexes by $\ch$. 

The point of this section is to provide a short review of the preliminary concepts, and notations, which are foundational to this paper . It is all standard except the last Section~\ref{subsec-character duality} which summarizes needed facts from~\cite{bravo-gillespie-hovey} and~\cite{bravo-gillespie}. Also, the useful Lemma~\ref{lemma-transfinite extensions of spheres and disks} has, to the author's knowledge, not appeared in the literature.

\subsection{Cotorsion pairs and precovers}\label{sec-cot} Let $\cat{A}$ be an abelian category.  By definition, a pair of classes $(\class{X},\class{Y})$ in $\cat{A}$ is called a \emph{cotorsion pair} if $\class{Y} = \rightperp{\class{X}}$ and $\class{X} = \leftperp{\class{Y}}$. Here, given a class of objects $\class{C}$ in $\cat{A}$, the right orthogonal  $\rightperp{\class{C}}$ is defined to be the class of all objects $X$ such that $\Ext^1_{\cat{A}}(C,X) = 0$ for all $C \in \class{C}$. Similarly, we define the left orthogonal $\leftperp{\class{C}}$. We call the cotorsion pair \emph{hereditary} if $\Ext^i_{\cat{A}}(X,Y) = 0$ for all $X \in \class{X}$, $Y \in \class{Y}$, and $i \geq 1$. The cotorsion pair is \emph{complete} if it has enough injectives and enough projectives. This means that for each $A \in \cat{A}$ there exist short exact sequences $0 \xrightarrow{} A \xrightarrow{} Y \xrightarrow{} X \xrightarrow{} 0$ and $0 \xrightarrow{} Y' \xrightarrow{} X' \xrightarrow{} A \xrightarrow{} 0$ with $X,X' \in \class{X}$ and $Y,Y' \in \class{Y}$.
Standard references include~\cite{enochs-jenda-book} and~\cite{trlifaj-book} and connections to abelian model categories can be found in~\cite{hovey} and~\cite{gillespie-hereditary-abelian-models}.

Complete cotorsion pairs are closely related to the study of precovers and pre-envelopes. This area has been extensively studied by many authors, especially Enochs, Jenda, Estrada, Garc\'\i{}a-Rozas, and many coauthors; For example, see~\cite{enochs-jenda-book, garcia-rozas}. Let $\class{X}$ be a class of objects in $\cat{A}$. A morphism  $\phi : X \xrightarrow{} A$ 
  in $\cat{A}$ is called an
  \emph{$\class{X}$-precover} if $X \in \class{X}$
  and $$\Hom_{\cat{A}}(X',X) \xrightarrow{} \Hom_{\cat{A}} (X',A) \xrightarrow{} 0$$ is exact
  for every $X' \in \class{X}$. Further, if $\ker{\phi} \in \rightperp{\class{X}}$, then $\phi$ is called a \emph{special $\class{X}$-precover.} Their is a dual notion of a \emph{(special) $\class{X}$-pre-envelope}. The connection to cotorsion pairs is the easy observation that if $(\class{X},\class{Y})$ is a complete cotorsion pair, then each object $A \in \cat{A}$ has a special $\class{X}$-precover and a special $\class{Y}$-pre-envelope.

\subsection{Projective and injective cotorsion pairs} 
Assume $\cat{A}$ is a bicomplete abelian category with enough projectives. By a \emph{projective cotorsion pair}  in $\cat{A}$ we mean a complete cotorsion pair $(\class{C},\class{W})$ for which $\class{W}$ is thick and $\class{C} \cap \class{W}$ is the class of projective objects. Such a cotorsion pair is equivalent to a \emph{projective model structure} on $\cat{A}$. By this we mean the model structure is abelian in the sense of~\cite{hovey} and all objects are fibrant. The cofibrant objects are exactly those in $\class{C}$ and the trivial objects are exactly those in $\class{W}$.  We also have the dual notion of \emph{injective cotorsion pairs} $(\class{W},\class{F})$ which give us \emph{injective model structures} on abelian categories with enough projectives. See~\cite{gillespie-recollement} for more on projective and injective cotorsion pairs. One important fact is that such cotorsion pairs are always hereditary and this implies that the associated homotopy category must be stable; that is, it is not just pre-triangulated but a triangulated category. We will use the following proposition to construct projective cotorsion pairs in this paper. 

\begin{proposition}[Construction of a projective model
structure]\label{prop-how to create a projective model structure}
Let $\cat{A}$ be a bicomplete abelian category with enough projectives
and denote the class of projectives by $\class{P}$. Let $\class{C}$ be
any class of objects and set $\class{W} =
\rightperp{\class{C}}$. Suppose the following conditions hold:
\begin{enumerate}
\item $(\class{C},\class{W})$ is a complete cotorsion pair.
\item $\class{W}$ is thick.
\item $\class{P} \subseteq \class{W}$.
\end{enumerate}
Then there is an abelian model structure on $\cat{A}$ where every
object is fibrant, $\class{C}$ are the cofibrant objects, $\class{W}$
are the trivial objects, and $\class{P} = \class{C} \cap \class{W}$
are the trivially cofibrant objects. In other words, $(\class{C},\class{W})$ is a projective cotorsion pair. 
\end{proposition}

\subsection{Chain complexes on abelian categories}\label{subsec-complexes}
Let $\cat{A}$ be an abelian category. We denote the corresponding category of chain complexes by $\cha{A}$. In the case $\cat{A} = \rmod$, we denote it by $\ch$. Our convention is that the differentials of our chain complexes lower degree, so $\cdots
\xrightarrow{} X_{n+1} \xrightarrow{d_{n+1}} X_{n} \xrightarrow{d_n}
X_{n-1} \xrightarrow{} \cdots$ is a chain complex. We also have the chain homotopy category of $\cat{A}$, denoted $K(\cat{A})$. Its objects are also chain complexes but its morphisms are chain homotopy classes of chain maps.
Given a chain complex $X$, the
\emph{$n^{\text{th}}$ suspension of $X$}, denoted $\Sigma^n X$, is the complex given by
$(\Sigma^n X)_{k} = X_{k-n}$ and $(d_{\Sigma^n X})_{k} = (-1)^nd_{k-n}$.
For a given object $A \in \cat{A}$, we denote the \emph{$n$-disk on $A$} by $D^n(A)$. This is the complex consisting only of $A \xrightarrow{1_A} A$ concentrated in degrees $n$ and $n-1$, and 0 elsewhere. We denote the \emph{$n$-sphere on $A$} by $S^n(A)$, and this is the complex consisting only of $A$ in degree $n$ and 0 elsewhere.

Given two chain complexes $X, Y \in \cha{A}$ we define $\homcomplex(X,Y)$ to
be the complex of abelian groups $ \cdots \xrightarrow{} \prod_{k \in
\Z} \Hom(X_{k},Y_{k+n}) \xrightarrow{\delta_{n}} \prod_{k \in \Z}
\Hom(X_{k},Y_{k+n-1}) \xrightarrow{} \cdots$, where $(\delta_{n}f)_{k}
= d_{k+n}f_{k} - (-1)^n f_{k-1}d_{k}$.
We get a functor
$\homcomplex(X,-) \mathcolon \cha{A} \xrightarrow{} \textnormal{Ch}(\Z)$. Note that this functor takes exact sequences to left exact sequences,
and it is exact if each $X_{n}$ is projective. Similarly the contravariant functor $\homcomplex(-,Y)$ sends exact sequences to left exact sequences and is exact if each $Y_{n}$ is injective. It is an exercise to check that the homology satisfies $H_n[Hom(X,Y)] = K(\cat{A})(X,\Sigma^{-n} Y)$.

Being an abelian category, $\cha{A}$ comes with Yoneda Ext groups. In particular, $\Ext^1_{\cha{A}}(X,Y)$ will denote the group of (equivalences classes) of short exact sequences $0 \xrightarrow{} Y \xrightarrow{} Z \xrightarrow{} X \xrightarrow{} 0$ under the Baer sum operation. There is a subgroup $\Ext^1_{dw}(X,Y) \subseteq \Ext^1_{\cha{A}}(X,Y)$ consisting of the ``degreewise split'' short exact sequences. That is,
those for which each $0 \xrightarrow{} Y_n \xrightarrow{} Z_n \xrightarrow{} X_n \xrightarrow{} 0$ is split exact. The following lemma gives a well-known connection between $\Ext^1_{dw}$ and the above hom-complex $\homcomplex$.

\begin{lemma}\label{lemma-homcomplex-basic-lemma}
For chain complexes $X$ and $Y$, we have isomorphisms:
$$\Ext^1_{dw}(X,\Sigma^{(-n-1)}Y) \cong H_n \homcomplex(X,Y) =
K(\cat{A})(X,\Sigma^{-n} Y)$$ In particular, for chain complexes $X$ and $Y$, $\homcomplex(X,Y)$ is
exact iff for any $n \in \Z$, any chain map $f \mathcolon \Sigma^nX \xrightarrow{} Y$ is
homotopic to 0 (or iff any chain map $f \mathcolon X \xrightarrow{} \Sigma^nY$ is homotopic
to 0).
\end{lemma}

In the case of $\cat{A} = R\textnormal{-Mod}$, we recall the usual tensor product of chain complexes. Given that $X$ (resp. $Y$) is a complex of right (resp. left) $R$-modules, the tensor product $X
\otimes Y$ is defined by $(X \otimes Y)_n = \oplus_{i+j=n} (X_i
\otimes Y_j)$ in degree $n$. The boundary map $\delta_n$ is defined
on the generators by $\delta_n (x \otimes y) = dx \otimes y +
(-1)^{|x|} x \otimes dy$, where $|x|$ is the degree of the element
$x$.

\subsection{Grothendieck categories}
Recall that a \emph{Grothendieck category} $\cat{G}$ is a cocomplete abelian category with a set of generators and such that direct limits are exact. Grothendieck categories automatically have enough injectives, and so such categories often admit injective cotorsion pairs yielding injective model structures on $\cat{G}$. If $\cat{G}$ possesses a set of projective generators then we can also look for projective cotorsion pairs in $\cat{G}$. In this paper we will be working with categories of $R$-modules, chain complexes of $R$-modules, and bicomplexes of $R$-modules. These are all Grothendieck categories possessing a set of projective generators.

\subsection{Disks and spheres and cotorsion pairs} Let $\cat{G}$ be a Grothendieck category. We point out a lemma that is often useful for constructing chain complexes in one side of a given cotorsion pair in $\cha{G}$. Recall that we say an object $M \in \cat{G}$ is a \emph{transfinite extension} of a set of objects $\class{S}$ when there is an ordinal $\lambda$ and $M = \varinjlim_{\alpha < \lambda} M_{\alpha}$ for some $\lambda$-diagram of monomorphisms $$M_0 \xrightarrow{i_0} M_1 \xrightarrow{i_1} \cdots M_{\alpha} \xrightarrow{i_{\alpha}} M_{\alpha +1} \xrightarrow{} \cdots $$ having $M_0 , \cok{i_{\alpha}} \in \class{S}$ for each $\alpha < \lambda$ and such that $M = \varinjlim_{\alpha < \gamma} M_{\alpha}$ for each limit ordinal $\gamma < \lambda$. It is well known that the left half of a cotorsion pair is closed under transfinite extensions and this is known as the Eklof Lemma. The dual statement is also true. We say an object $M$ is an \emph{inverse transfinite extension} of a set of objects $\class{S}$ when $M = \varprojlim_{\alpha < \lambda} M_{\alpha}$ for some $\lambda$-diagram of surjections $$M_0 \xleftarrow{i_0} M_1 \xleftarrow{i_1} \cdots M_{\alpha} \xleftarrow{i_{\alpha}} M_{\alpha +1} \xleftarrow{} \cdots $$ having $M_0 , \ker{i_{\alpha}} \in \class{S}$ for each $\alpha < \lambda$ and such that $M = \varprojlim_{\alpha < \gamma} M_{\alpha}$ for each limit ordinal $\gamma < \lambda$. It was shown in~\cite[Lemma~2.3]{trlifaj-Ext and inverse limits} that the right half of a cotorsion pair in $R$-Mod is closed under inverse transfinite extensions. 
These ideas are applied to get the following lemma.

\begin{lemma}\label{lemma-transfinite extensions of spheres and disks}
Let $\cat{G}$ be a Grothendieck category with a projective generator and let $(\class{X},\class{Y})$ be a cotorsion pair of chain complexes in $\cha{G}$. Suppose $\class{C}$ is some given class of objects in $\cat{G}$.
\begin{enumerate}
\item If the spheres $S^n(C)$ are in $\class{X}$ whenever $C$ is in $\class{C}$, then any bounded below complex with entries in $\class{C}$ is also in $\class{X}$.
\item If the disks $D^n(C)$ are in $\class{X}$ whenever $C$ is in $\class{C}$, then any bounded above exact complex with cycles in $\class{C}$ is also in $\class{X}$.
\item If the spheres $S^n(C)$ are in $\class{Y}$ whenever $C$ is in $\class{C}$, then any bounded above complex with entries in $\class{C}$ is also in $\class{Y}$.
\item If the disks $D^n(C)$ are in $\class{Y}$ whenever $C$ is in $\class{C}$, then any bounded below exact complex with cycles in $\class{C}$ is also in $\class{Y}$.
\end{enumerate}
\end{lemma}

\begin{proof}
Note that (1) and (3) are dual statements and (2) and (4) are dual. We will prove (1) and (4). For (1), suppose that $(X,d)$ is a bounded below complex with entries in $\class{C}$. It is easy to check that $X$ can be expressed as a transfinite extension of spheres $S^n(X_n)$, on the components $X_n$. Each $S^n(X_n)$ is in $\class{X}$ by hypothesis and so $X$ is in $\class{X}$ too by the Eklof Lemma.

Next we prove (4). Here we note that any bounded below exact complex $(X,d)$ can be expressed as an inverse transfinite extension as indicated in the diagram:
$$\begin{CD}
             @.             @.      0  @<<< Z_3X  @<<< \cdots \\
  @.         @.                    @VVV         @VVV      \\
             @.        0     @<<<    Z_2X    @<<d< X_3  @= \cdots \\
  @.         @VVV                    @VVV         @VdVV      \\
  0           @<<<  Z_1X           @<<d<    X_2    @= X_2  @= \cdots \\
  @VVV         @VVV                    @VdVV         @VdVV      \\
  X_0 @<<d< X_1            @=  X_1           @=    X_1    @= \cdots   \\
  @|         @VdVV                    @VdVV         @VdVV      \\
  X_0 @= X_0            @=  X_0           @=    X_0    @= \cdots   \\
  @VVV         @VVV                    @VVV         @VVV      \\
  0   @.       0              @.        0              @.        0        @. \\
\end{CD}$$ Indeed note that each horizontal map in the diagram is surjective with its kernel being a disk $D^{n+1}(Z_nX)$. So $X$ is an inverse transfinite extension of the disks $D^{n+1}(Z_nX)$. The desired result now follows from~\cite[Lemma~2.3]{trlifaj-Ext and inverse limits} which is the dual of the Eklof Lemma. The proof of~\cite[Lemma~2.3]{trlifaj-Ext and inverse limits} is given for the category of modules over a ring, but the proof holds in any Grothendieck category with a projective generator.
\end{proof}

\subsection{The modified Hom and Tensor complexes}\label{subsec-modified hom and tensor} Here we focus in particular on $\ch$, the category of chain complexes of $R$-modules. The above $\homcomplex$ of Section~\ref{subsec-complexes} is often referred to as the \emph{internal hom}, for in the case that $R$ is commutative, $\homcomplex(X,Y)$ is again an object of $\ch$. Note that the cycles in degree 0 of the internal hom coincide with the \emph{external hom} functor: $Z_0[\homcomplex(X,Y)] \cong \Hom_{\ch}(X,Y)$. This idea can be used to define an alternate internal hom as was done in~\cite{enochs-garcia-rozas} and~\cite{garcia-rozas}. (This is the hom that corresponds to the graded hom in the category of graded $R[x]/(x^2)$-modules, where $R[x]/(x^2)$ is thought of as a graded ring with a copy of $R$ in degrees 0 and $-1$, and putting $x$ in degree $-1$.)
To define it for a given pair $X, Y \in \ch$, we let $\overline{\homcomplex}(X,Y)$ to be the complex $$\overline{\homcomplex}(X,Y)_n = Z_n\homcomplex(X,Y)$$ with differential $$\lambda_n : \overline{\homcomplex}(X,Y)_n \xrightarrow{} \overline{\homcomplex}(X,Y)_{n-1}$$ defined by $(\lambda f)_k = (-1)^nd_{k+n}f_k$. Notice that the degree $n$ component of $\overline{\homcomplex}(X,Y)$ is exactly $\Hom_{\ch}(X,\Sigma^{-n}Y)$. In this way we get an internal hom $\overline{\homcomplex}$ which is useful for categorical considerations in $\ch$. For example, $\overline{\homcomplex}(X,-)$ is a left exact functor, and is exact if and only if $X$ is projective in the category $\ch$. On the other hand, $\overline{\homcomplex}(-,Y)$ is exact if and only if $Y$ is injective in $\ch$. There are corresponding derived functors which we denote by $\overline{\mathit{Ext}}^i$. They satisfy that $\overline{\mathit{Ext}}^i(X,Y)$ is a complex whose degree $n$ is $\Ext^i_{\ch}(X,\Sigma^{-n}Y)$.

Similarly, the usual tensor product of chain complexes does not characterize categorical flatness. For this one needs the modified tensor product and its left derived torsion functor from~\cite{enochs-garcia-rozas} and~\cite{garcia-rozas}. We will denote it by $\overline{\otimes}$, and it is defined in terms of the usual tensor product $\otimes$ as follows. Given a complex $X$ of right $R$-modules and a complex $Y$ of left $R$-modules, we define $X \overline{\otimes} Y$ to be the complex whose $n^{\text{th}}$ entry is $(X \otimes Y)_n / B_n(X \otimes Y)$ with boundary map  $(X \otimes Y)_n / B_n(X \otimes Y) \rightarrow (X \otimes Y)_{n-1} / B_{n-1}(X \otimes Y)$ given by
\[
\overline{x \otimes y} \mapsto \overline{dx \otimes y}.
\]
This defines a complex and we get a bifunctor $ - \overline{\otimes} - $ which is right exact in each variable. We denote the corresponding left derived functors by $\overline{\Tor}_i$. We refer the reader to~\cite{garcia-rozas} for more details.

\subsection{Finitely chain complexes and projective chain complexes} A standard characterization of projective objects in $\ch$ is the following: A complex $P$ is \emph{projective} if and only if it is an exact complex with each cycle $Z_nP$ a projective $R$-module. We also recall that, by definition, a chain complex $X$ is \emph{finitely generated} if whenever $X = \Sigma_{i \in I} S_i$, for some collection $\{S_i\}_{i \in I}$ of subcomplexes of $X$, then there exists a finite subset $J \subseteq I$ for which $X = \Sigma_{i \in J} S_i$. It is a standard fact that $X$ is finitely generated if and only if it is bounded (above and below) and each $X_n$ is finitely generated. We say that a chain complex $X$ is of \textbf{type $\boldsymbol{FP_{\infty}}$} if it has a projective resolution by finitely generated projective complexes. Certainly any such $X$ is finitely presented and hence finitely generated.
Recall that by definition a chain complex $X$ is \emph{finitely presented} if $\Hom_{\ch}(X,-)$ preserves direct limits; $X$ is finitely presented if and only if it is bounded and each $X_n$ is a finitely presented $R$-module.

\subsection{Absolutely clean and level complexes; character duality}\label{subsec-character duality}
The so-called level and absolutely clean $R$-modules were introduced in~\cite{bravo-gillespie-hovey} as generalizations of flat modules over coherent rings and injective modules over Noetherian rings. The same notions in the category $\ch$ were also studied in~\cite{bravo-gillespie}. Here we recall some definitions and results from~\cite{bravo-gillespie} that will be used in the present paper.

\begin{definition}
We call a chain complex $A$ \emph{absolutely clean} if $\Ext^1_{\ch}(X,A)=0$ for all chain complexes $X$ of type $FP_{\infty}$. Equivalently, if $\overline{\mathit{Ext}}^1(X,A) = 0$ for all complexes $X$ of type $FP_{\infty}$. On the other hand, we call a chain complex $L$ \emph{level} if $\overline{\Tor}_1(X,L) = 0$ for all chain complexes $X$ of right $R$-modules of type $FP_{\infty}$.
\end{definition}

For the reader's convenience we now list some properties of the absolutely clean and level complexes.

\begin{proposition}\cite[Propositions 2.6 and 4.6]{bravo-gillespie}\label{prop-level chain complexes}
A chain complex $A$ is absolutely clean if and only if $A$ is exact and each $Z_nA$ is an absolutely clean $R$-module.
A chain complex $L$ is level if and only if $L$ is exact and each $Z_nL$ is a level $R$-module.
\end{proposition}

Recall that the character module of $M$ is defined as $M^+ = \Hom_{\Z}
(M, \Q )$, and that $M^+$ is a right (resp. left) $R$-module whenever $M$ is a left (resp. right) $R$-module.  The construction extends to chain complexes: Given a chain complex $X$, we have $X^+ = \Hom_{\Z}(X,\Q)$. Since $\Q$ is an injective cogenerator for the category of abelian groups, the functor $\Hom_{\Z}(-,\Q)$ preserves and reflects exactness. So Proposition~\ref{prop-level chain complexes} immediately gives us the following corollary due to the perfect character module duality between absolutely clean and level modules~\cite[Theorem~2.10]{bravo-gillespie-hovey}.

\begin{proposition}\cite[Corollary 4.7]{bravo-gillespie}\label{cor-duality}
A chain complex $L$ of left (resp. right) modules is level if and only if $L^+ = \Hom_{\Z}(L,\Q)$ is an absolutely clean complex of right (resp. left) modules. And, a chain complex $A$ of left (resp. right) modules is absolutely clean if and only if $A^+ = \Hom_{\Z}(A,\Q)$ is a level complex of right (resp. left) modules.
\end{proposition}

The notion of duality pair used in~\cite{bravo-gillespie-hovey} was extended to chain complexes in~\cite{gillespie-ding-modules}. We recall the definition: Suppose $\cat{C}$ is a collection of chain complexes of right $R$-modules, and $\cat{D}$ is a collection of chain complexes of left $R$-modules, we say that $(\cat{C},\cat{D})$ is a \emph{duality pair} if $X \in \cat{C}$ if
and only if ${X}^{+} \in  \cat{D}$, and $Y \in \cat{D}$ if and
only if $Y^{+} \in \cat{C}$.  It is immediate from Corollary~\ref{cor-duality} that the absolutely clean and level complexes give rise to two duality pairs. One where $\class{C}$ is the class of all absolutely clean complexes of right $R$-modules, and another where $\class{C}$ is the class of all level complexes of right $R$-modules.

\begin{proposition}\cite[Theorem 5.9]{gillespie-ding-modules}\label{prop-dual-exact}
Suppose $(\cat{C},\cat{D})$ is a duality pair in $\ch$
such that $\cat{D}$ is closed under pure quotients.  Let $\mathbb{C}$ be a chain 
complex of projective chain complexes.  Then $X \overline{\otimes} \mathbb{C}$ is exact for all $X \in
\cat{C}$ if and only if $\overline{\homcomplex}(\mathbb{C},Y)$ is exact for all $Y \in
\cat{D}$.  In particular,  $A \overline{\otimes} \mathbb{C}$ is exact for all absolutely clean complexes $A$ if and only if $\overline{\homcomplex}(\mathbb{C},L)$ is exact for all level complexes $L$.
\end{proposition}

The classes of absolutely clean and level complexes each possess a long list of nice homological properties. For example, each is closed under direct products, direct sums, direct summands, direct limits, transfinite extensions, pure submodules and pure quotients. Moreover, the level complexes form a resolving class while the absolutely clean complexes form a coresolving class; see~\cite[Propositions~2.7 and 4.8]{bravo-gillespie}. One of the most important properties for our purposes is listed in the following proposition.

\begin{proposition}\cite[Corollaries 2.11 and 4.9]{bravo-gillespie}\label{prop-level complexes are transfinite extensions}
There exists a cardinal $\kappa$ such that every absolutely clean (resp. level) chain complex is a transfinite extension of absolutely clean (resp. level) complexes with cardinality bounded by $\kappa$. In particular, there is a set $\class{S}$ of absolutely clean (resp. level) complexes for which every absolutely clean (resp. level) complex is a transfinite extension of ones in $\class{S}$.
\end{proposition}

\section{Projective model structures on double complexes}\label{sec-proj}

Since $\ch$ is an abelian category we can of course consider  $\textnormal{Ch}(\ch)$, the category of chain complexes of chain complexes. Using~\cite[Sign Trick~1.2.5]{weibel}, the category $\textnormal{Ch}(\ch)$ can be identified with the category of bicomplexes. However, for our purpose here it is easier to stick with the category $\textnormal{Ch}(\ch)$, and we will refer to an object in this category as either a \emph{double complex} or a \emph{complex of complexes}. Another way the reader may wish to think about this category is to first identify $\ch$ with the category $R[x]/(x^2)-$Mod, of graded $R[x]/(x^2)-$modules over the graded ring $R[x]/(x^2)$ (putting $x$ in degree $-1$). Then the category of double complexes we work with may be identified with $\textnormal{Ch}(R[x]/(x^2)-\textnormal{Mod})$, the category of chain complexes of graded $R[x]/(x^2)-$modules. The paper~\cite{gillespie-hovey-graded-gorenstein} has more details on this perspective for the interested reader. 

The purpose of this section is to prove the following theorem, which is a chain complex version of~\cite[Theorem~6.1]{bravo-gillespie-hovey}.

\begin{theorem}\label{thm-how to create projective on chain}
Given a ring $R$, let $A$ be a given chain complex of right $R$-modules.  Let
$\class{C}$ be the class of all $A$-acyclic complexes of projective complexes; that
is, chain complexes $\mathbb{C}$ with each $\mathbb{C}_n$ a projective chain complex and such that
$A \overline{\otimes} \mathbb{C}$ is exact.  Then there is a cofibrantly generated
abelian model structure on $\textnormal{Ch}(\ch)$ where every object is fibrant,
$\class{C}$ is the class of cofibrant objects, and $\class{W} =
\rightperp{\class{C}}$ is the class of trivial objects. In other words, $(\class{C}, \rightperp{\class{C}})$ is a projective cotorsion pair in $\textnormal{Ch}(\ch)$.
\end{theorem}

To prove Theorem~\ref{thm-how to create projective on chain} we follow the sequence of lemmas from~\cite[Section~7]{bravo-gillespie-hovey}, extending  them to work for double complexes rather than just chain complexes. The proofs are essentially the same but we include the general versions here for clarity and convenience of the reader. Again, the key is to resist the temptation to work with bicomplexes and to note that the arguments readily adapt to working with double complexes (complexes of graded $R[x]/(x^2)$-modules). 

We start with a classic result of Kaplansky~\cite{kaplansky-projective} stating that every projective module is a direct sum of countably generated
projective modules. It follows that the same result holds for a projective chain complex too, which we explain in the following lemma. 

\begin{lemma}[Kaplansky]\label{lemma-Kaplansky for chain complexes}
The following are equivalent for a chain complex $P$. 
\begin{enumerate}
\item $P$ is projective in $\ch$.
\item $P$ is a direct sum of countably generated projective complexes.
\item $P \cong \oplus_{i \in I} D^{n_i}(P_i)$ for some countably generated projective $R$-modules $P_i$.
\end{enumerate}
\end{lemma} 

\begin{proof}
We note that any chain complex $X$ is countably generated if and only if each $X_n$ is countably generated (for example, see~\cite[Lemma~4.10]{gillespie-quasi-coherent}, taking $\kappa = \aleph_1$).
The implications $(3) \implies (2) \implies (1)$ are clear. For $(1) \implies (3)$, it is well known that a projective complex is isomorphic to a direct sum $\oplus_{n \in \Z} D^n(P_n)$ where each $P_n$ is some projective $R$-module. But the classic result of Kaplansky~\cite{kaplansky-projective} tells us that each projective $P_n$ is in turn a direct sum of countably generated projectives. So (3) is clear too. 
\end{proof} 

\begin{definition}\label{def-cardinality}
We define the \textbf{cardinality} of a chain complex $X$ of $R$-modules to be $|\coprod_{n \in \Z} X_n|$. The cardinality of a double chain complex $\mathbb{X} \in \textnormal{Ch}(\ch)$ is defined similarly. 
\end{definition}

\begin{lemma}[Covering Lemma for double complexes]\label{lemma-covering lemma}
Let $\kappa$ be an infinite cardinal and suppose $\mathbb{X}$ is a nonzero
double complex in which each $\mathbb{X}_n$ has a direct sum decomposition $\mathbb{X}_n =
\oplus_{i \in I_n} M_{n,i}$ where each chain complex $M_{n,i}$ has $|M_{n,i}| < \kappa$ for all $i \in
I_n$. Then for any choice of subcollections $J_n \subseteq I_n$ (at
least one of which is nonempty), with $|J_n| < \kappa$, we can find a
nonzero subcomplex $\mathbb{S} \subseteq \mathbb{X}$ with each $\mathbb{S}_n = \oplus_{i \in K_n}
M_{n,i}$ for some subcollections $K_n \subseteq I_n$ satisfying $J_n
\subseteq K_n$ and $|K_n| < \kappa$.
\end{lemma}

\begin{proof}
Suppose we are given such subcollections $J_n \subseteq I_n$. First,
for each $n$, we may build a subcomplex $\mathbb{X}^n$ of $\mathbb{X}$ as follows: In
degree $n$ the (double) complex will consist of $\oplus_{i \in J_n}
M_{n,i}$. Then noting $d(\oplus_{i \in J_n} M_{n,i}) \subseteq
\oplus_{i \in I_{n-1}} M_{n-1,i}$  and $|d(\oplus_{i \in J_n} M_{n,i})| < \kappa$, 
we can find a subset $L_{n-1} \subseteq I_{n-1}$ such that $|L_{n-1}| < \kappa$
and yet $d(\oplus_{i \in J_n} M_{n,i}) \subseteq \oplus_{i \in L_{n-1}} M_{n-1,i}$.
Now the subcomplex of $\mathbb{X}$ that 
we are constructing will consist of $\oplus_{i \in L_{n-1}} M_{n-1,i}$
in degree $n-1$. We continue down in the same way finding $L_{n-2}
\subseteq I_{n-2}$ with $|L_{n-2}| < \kappa$ and with $d(\oplus_{i \in
L_{n-1}} M_{n-1,i}) \subseteq \oplus_{i \in L_{n-2}} M_{n-2,i}$. In
this way we get a subcomplex of $\mathbb{X}$: $$\mathbb{X}^n = \cdots \xrightarrow{} 0
\xrightarrow{} \oplus_{i \in J_{n}} M_{n,i} \xrightarrow{} \oplus_{i
\in L_{n-1}} M_{n-1,i} \xrightarrow{} \oplus_{i \in L_{n-2}} M_{n-2,i}
\xrightarrow{} \cdots $$ Finally set $\mathbb{S} = \Sigma_{l \in \N} \mathbb{X}^l$ and note
that this double complex, obviously nonzero because at least one $J_n \neq
\phi$, will work. (The sets $K_n$ we claim to exist are the union of
all the $J_n$'s and all the various $L_i$ in sight. We still have
$|K_n| < \kappa$.)
\end{proof}

Now we have a similar lemma but concerning exact complexes of chain complexes.

\begin{lemma}[Exact Covering Lemma for double complexes]\label{lemma-exact covering lemma}
Let $\kappa$ be an infinite cardinal and suppose $\mathbb{Y}$ is an exact
complex of chain complexes in which each $\mathbb{Y}_n$ has a direct sum decomposition $\mathbb{Y}_n =
\oplus_{i \in I_n} M_{n,i}$ where each chain complex $M_{n,i}$ has $|M_{n,i}| < \kappa$ for all $i \in
I_n$. Then for any choice of subcollections $K_n \subseteq I_n$, with
$|K_n| < \kappa$, we can find an exact subcomplex $\mathbb{T} \subseteq \mathbb{Y}$ with
each $\mathbb{T}_n = \oplus_{i \in J_n} M_{n,i}$ for some subcollections $J_n
\subseteq I_n$ satisfying $K_n \subseteq J_n$ and $|J_n| < \kappa$.
\end{lemma}

\begin{proof}
We prove this in two steps.

\noindent (Step 1). We first show the following: If $\mathbb{X} \subseteq \mathbb{Y}$ is
any exact subcomplex with $|\mathbb{X}| < \kappa$, then for any single one of
the given $K_n$, we can find an exact subcomplex $\mathbb{T} \subseteq \mathbb{Y}$
containing $\mathbb{X}$ and so that for this given $n$, $\mathbb{T}_n = \oplus_{i \in
L_n} M_{n,i}$ for some $L_n \subseteq I_n$ with $K_n \subseteq L_n$
and $|L_n| < \kappa$.

First, we can find for the given $n$ (since $|\mathbb{X}_n| < \kappa$), a subset $D_n \subseteq 
I_n$ with $|D_n| < \kappa$ such that $\mathbb{X}_n \subseteq \oplus_{i \in D_n} M_{n,i}$. 
Now define $L_n = D_n \cup K_n$ and set $\mathbb{T}_n = \oplus_{i \in
L_n} M_{n,i}$. Of course $|L_n| < \kappa$ and $\mathbb{X}_n \subseteq \mathbb{T}_n$.

So all we need to do is extend $\mathbb{T}_n$ into an exact subcomplex
containing $\mathbb{X}$ and with cardinality less than $\kappa$. We build down
by setting $\mathbb{T}_{n-1} = \mathbb{X}_{n-1} + d(\mathbb{T}_n)$ and $\mathbb{T}_i = \mathbb{X}_i$ for all $i <
n-1$. One can check that \[\mathbb{T}_n \xrightarrow{} \mathbb{X}_{n-1} + d(\mathbb{T}_n) \xrightarrow{} \mathbb{X}_{n-2} \xrightarrow{}
\cdots\] is exact. In particular, we have exactness in degree $n-1$
since $d(\mathbb{X}_n) \subseteq d(\mathbb{T}_n)$.

Next we build up from $\mathbb{T}_n$. To start, take the kernel of $\mathbb{T}_n \xrightarrow{}
\mathbb{T}_{n-1}$ and find a $\mathbb{T}'_{n+1} \subseteq \mathbb{Y}_{n+1}$ such that $|\mathbb{T}'_{n+1}|
< \kappa$ and $\mathbb{T}'_{n+1}$ maps surjectively onto this kernel. Then take
$\mathbb{T}_{n+1} = \mathbb{X}_{n+1} + \mathbb{T}'_{n+1}$. Now $\mathbb{T}_{n+1}$ also maps surjectively
onto this kernel. We continue upward to build $\mathbb{T}_{n+2}, \mathbb{T}_{n+3},
\cdots$ in the same way and we are done.

\

\noindent (Step 2). We now finish the proof. From Step 1, taking $\mathbb{X} = 0$ and the
subcollection to be $K_0$ we can find an exact subcomplex $\mathbb{T}^0
\subseteq \mathbb{Y}$ such that $(\mathbb{T}^0)_0 = \oplus_{i \in L_0} M_{0,i}$ for some
$L_0 \subseteq I_0$ with $K_0 \subseteq L_0$ and $|L_0| < \kappa$.
Now using Step 1 again, with $\mathbb{X} = \mathbb{T}^0$ and using $K_{-1}$, we get
another exact subcomplex $\mathbb{T}^1$ containing $\mathbb{T}^0$ and such that
$(\mathbb{T}^1)_{-1} = \oplus_{i \in L_{-1}} M_{{-1},i}$ for some $L_{-1}
\subseteq I_{-1}$ with $K_{-1} \subseteq L_{-1}$ and $|L_{-1}| <
\kappa$. Lets say that $\mathbb{T}^0$ was constructed using a ``degree 0
operation'' and $\mathbb{T}^1$ was constructed using a ``degree -1
operation''. Then we can continue to use ``degree $k$ operations''
with the following back and forth pattern on $k$:
\[0, \ \ -1,0,1, \ \ -2,-1,0,1,2, \ \ -3, -2, -1, 0, 1,2,3 \ \
\cdots\] to build an increasing chain of exact subcomplexes, $\{\, \mathbb{T}^l
\,\}$.  Finally set $\mathbb{T} = \cup_{l \in \N} \mathbb{T}^l$. Then by a cofinality
argument we see that for each $n$ we have $\mathbb{T}_n = \oplus_{i \in J_n}
M_{n,i}$ for some subsets $J_n \subseteq I_n$ (the $J_n$'s are each a
countable union of the newly constructed $L_n$'s obtained in each
``pass'', and so $|J_n| < \kappa$). Clearly each $K_n \subseteq J_n$
and $\mathbb{T}$ is an exact subcomplex of $\mathbb{Y}$.
\end{proof}

With these lemmas in hand, we now return to the hypotheses of Theorem~\ref{thm-how to create projective on chain}. First suppose that we are given a chain complex $\mathbb{P} \in \textnormal{Ch}(\ch)$ of projective chain complexes. By the Kaplansky result, Lemma~\ref{lemma-Kaplansky for chain complexes}, we can write each component $\mathbb{P}_n$ as a direct sum $\mathbb{P}_n = \oplus_{i
\in I_n} P_{n,i}$ where each $P_{n,i}$ is a countably generated projective chain complex. Note that if $\kappa > \text{max}\{\,
|R| \, , \, \omega \,\}$ is a regular cardinal, then $|P_{n,i}| <
\kappa$.

Next, referring again to the hypotheses of Theorem~\ref{thm-how to create projective on chain}, assume we are given a chain complex $A$ of right $R$-modules. Using the natural isomorphism (see~\cite[Prop.~4.2.1]{garcia-rozas})
\[
X \overline{\otimes} (\oplus_{i \in \class{S}} Y_i) \cong \oplus_{i \in
\class{S}} (X \overline{\otimes} Y_i)
\]
we may \emph{identify} $A \overline{\otimes} \,\mathbb{P}$ with the complex whose degree
$n$ is $\oplus_{i \in I_n} A \overline{\otimes} P_{n,i}$. Moreover, for any
subcomplex $\mathbb{S} \subseteq \mathbb{P}$ of the form $\mathbb{S}_n = \oplus_{i \in K_n}
P_{n,i}$ for some $K_n \subseteq I_n$ we can and will identify $A
\overline{\otimes} \,\mathbb{S}$ with the subcomplex of $A \overline{\otimes} \,\mathbb{P}$ whose degree $n$
is $$\oplus_{i \in K_n} A \overline{\otimes} P_{n,i} \subseteq \oplus_{i \in
I_n} A \overline{\otimes} P_{n,i}.$$ We note that if $\kappa > \text{max}\{\,
|R| \, , \, \omega \,\}$ is a regular cardinal, then such a subcomplex
$\mathbb{S}$ satisfies $|\mathbb{S}| < \kappa$ whenever $|K_n| < \kappa$. Similarly, if
$\kappa > \text{max}\{\, |A| \, , \, \omega \,\}$ is a regular
cardinal, note that $|A
\overline{\otimes} \,\mathbb{S}| < \kappa$ whenever $|K_n| <
\kappa$. We will use all of the above observations in the proof of our
theorem below.

\begin{theorem}\label{theorem-filtrations for complexes of projectives}
Let $A$ be a given chain complex of right $R$-modules and take $\kappa > \text{max}\{\, |R| \,
, \, |A| \, , \, \omega \,\}$ to be a regular cardinal. Let $\mathbb{P}$ be any
nonzero complex of projective complexes in which $A \overline{\otimes} \,\mathbb{P}$ is exact. Then
we can write $\mathbb{P}$ as a continuous union $\mathbb{P} = \cup_{\alpha < \lambda}
\mathbb{Q}_{\alpha}$ where each $\mathbb{Q}_{\alpha}, \mathbb{Q}_{\alpha + 1}/\mathbb{Q}_{\alpha}$ are
also $A \overline{\otimes} \,-\,$exact complexes of projective complexes (that is, each is $A$-acyclic) and $|\mathbb{Q}_{\alpha}|,
|\mathbb{Q}_{\alpha + 1}/\mathbb{Q}_{\alpha}| < \kappa$.
\end{theorem}

\begin{proof}
As described before the statement of the theorem, we write each $\mathbb{P}_n = \oplus_{i \in I_n} P_{n,i}$ where each $P_{n,i}$ is a
countably generated projective complex. We prove the theorem in two steps.

\noindent (Step 1). We first show the following: We can find a nonzero
subcomplex $\mathbb{Q} \subseteq \mathbb{P}$ of the form $\mathbb{Q}_n = \oplus_{i \in L_n}
P_{n,i}$ for some subcollections $L_n \subseteq I_n$ having $|L_n| <
\kappa$ and such that $A \overline{\otimes} \,\mathbb{Q}$ is exact.

Since $\mathbb{P}$ is nonzero at least one $\mathbb{P}_n \neq 0$. For this $n$, take any
nonempty $J_n \subseteq I_n$ having $|J_n| < \kappa$. Apply the
Covering Lemma~\ref{lemma-covering lemma} with $\mathbb{P}$ in the place of $\mathbb{X}$ and taking the subcollections to consist of this $J_n$ and all the other $J_n$ may be
empty. This gives us a nonzero subcomplex with $\mathbb{S}^1_n = \oplus_{i \in
K^1_n} P_{n,i}$ for some subcollections $K^1_n \subseteq I_n$
satisfying $J_n \subseteq K^1_n$ and $|K^1_n| < \kappa$ for each $n$.

Now $A \overline{\otimes} \,\mathbb{S}^1$ is the subcomplex of $A \overline{\otimes} \,\mathbb{P}$ having $(A \overline{\otimes} \,\mathbb{S}^1)_n = \oplus_{i \in K^1_n} A \overline{\otimes} P_{n,i}$. That is,
the subcollections $K^1_n \subseteq I_n$ determine $A \overline{\otimes} \,\mathbb{S}^1$. We now apply the Exact Covering Lemma~\ref{lemma-exact covering lemma} with $A \overline{\otimes} \,\mathbb{P}$ in
the place of $\mathbb{Y}$ and taking the subcollections to be the $K^1_n$. This
gives us an exact subcomplex $\mathbb{T}^1 \subseteq A \overline{\otimes} \,\mathbb{P}$ with each
$\mathbb{T}^1_n = \oplus_{i \in J^1_n} A \overline{\otimes} P_{n,i}$ for some
subcollections $J^1_n \subseteq I_n$ satisfying $K^1_n \subseteq
J^1_n$ and $|J^1_n| < \kappa$.

But perhaps now the direct sums $\oplus_{i \in J^1_n} P_{n,i}$ don't
even form a \emph{subcomplex} of $\mathbb{P}$ (because the tensor product with
$A$ may send some maps to $0$). So we again apply the Covering Lemma
to $\mathbb{P}$ with the $J^1_n$ as the subcollections to find a subcomplex
$\mathbb{S}^2 \subseteq \mathbb{P}$ with each $\mathbb{S}^2_n = \oplus_{i \in K^2_n} P_{n,i}$ for
some subcollections $K^2_n \subseteq I_n$ satisfying $J^1_n \subseteq
K^2_n$ and $|K^2_n| < \kappa$. Of course $\mathbb{S}^1 \subseteq \mathbb{S}^2$ because
$K^1_n \subseteq K^2_n$ for each $n$.

But now certainly $A \overline{\otimes} \,\mathbb{S}^2$ need not be exact, so we again
apply the Exact Covering Lemma to $A \overline{\otimes} \,\mathbb{P}$ taking the
subcollections to be the $K^2_n$. This gives us an exact subcomplex
$\mathbb{T}^2 \subseteq A \overline{\otimes} \,\mathbb{P}$ with each $\mathbb{T}^2_n = \oplus_{i \in J^2_n}
A \overline{\otimes} P_{n,i}$ for some subcollections $J^2_n \subseteq I_n$
satisfying $K^2_n \subseteq J^2_n$ and $|J^2_n| < \kappa$. Notice that
we have $A \overline{\otimes} \,\mathbb{S}^1 \subseteq \mathbb{T}^1 \subseteq A \overline{\otimes} \,\mathbb{S}^2
\subseteq \mathbb{T}^2$ because $K^1_n \subseteq J^1_n \subseteq K^2_n
\subseteq J^2_n$.

But again, the $\oplus_{i \in J^2_n} P_{n,i}$ need not form a
subcomplex of $\mathbb{P}$. So we continue this back and forth method, applying the
Covering Lemma to $\mathbb{P}$ and the newly obtained subcollections $J^l_n$,
and then applying the Exact Covering Lemma to $A \overline{\otimes} \,\mathbb{P}$ and the
newly found subcollections $K^l_n$. We obtain an increasing sequence
of subcomplexes of $\mathbb{P}$ $$0 \neq \mathbb{S}^1 \subseteq \mathbb{S}^2 \subseteq \mathbb{S}^3
\subseteq \cdots$$ corresponding to the subcollections $J^1_n
\subseteq J^2_n \subseteq J^3_n \subseteq \cdots$. We also get an
increasing sequence of subcomplexes of $A \overline{\otimes} \,\mathbb{P}$
$$A \overline{\otimes} \,\mathbb{S}^1 \subseteq \mathbb{T}^1 \subseteq A \overline{\otimes} \,\mathbb{S}^2 \subseteq
\mathbb{T}^2 \subseteq A \overline{\otimes} \,\mathbb{S}^3 \subseteq  \mathbb{T}^3 \subseteq \cdots$$
with each $\mathbb{T}^l$ exact.

So we set $\mathbb{Q} = \cup_{l \in \N} \mathbb{S}^l$ and claim that $\mathbb{Q}$ satisfies the
properties we sought. Indeed notice each $\mathbb{Q}_n = \oplus_{i \in L_n}
P_{n,i}$ where $L_n = \cup_{l \in \N} J^l_n$. Also we still have
$|L_n| < \kappa$. Finally, since $A \overline{\otimes} \,-\,$ commutes with direct
limits we get $A \overline{\otimes} \,\mathbb{Q} = \cup_{l \in \N} A \overline{\otimes} \,\mathbb{S}^l
=\cup_{l \in \N} \mathbb{T}^l$. This complex is exact because each $\mathbb{T}^l$ is
exact.

\

\noindent (Step 2). We now can easily finish to obtain the desired
continuous union. Start by finding a nonzero $\mathbb{Q}^0 \subseteq \mathbb{P}$ of the
form $\mathbb{Q}^0_n = \oplus_{i \in L^0_n} P_{n,i}$ for some subcollections
$L^0_n \subseteq I_n$ having $|L^0_n| < \kappa$ and such that $A \overline{\otimes} \,\mathbb{Q}^0$ is exact. Note that $\mathbb{Q}^0$ and $\mathbb{P}/\mathbb{Q}^0$ are also
complexes of projective complexes and since $0 \xrightarrow{} \mathbb{Q}^0
\xrightarrow{} \mathbb{P} \xrightarrow{} \mathbb{P}/\mathbb{Q}^0 \xrightarrow{} 0$ is a
degreewise split short exact sequence, so must be $$0 \xrightarrow{} A \overline{\otimes} \,\mathbb{Q}^0 \xrightarrow{} A \overline{\otimes} \,\mathbb{P} \xrightarrow{} A \overline{\otimes} \,\mathbb{P}/\mathbb{Q}^0 \xrightarrow{} 0.$$ It follows that $A \overline{\otimes} \,\mathbb{P}/\mathbb{Q}^0$ must also
be exact. So if it happens that $\mathbb{P}/\mathbb{Q}^0$ is nonzero we can in turn find
a nonzero subcomplex $\mathbb{Q}^1/\mathbb{Q}^0 \subseteq \mathbb{P}/\mathbb{Q}^0$ with $\mathbb{Q}^1/\mathbb{Q}^0$ and
$$(\mathbb{P}/\mathbb{Q}^0)/(\mathbb{Q}^1/\mathbb{Q}^0) \cong \mathbb{P}/\mathbb{Q}^1$$ both $A \overline{\otimes} \,-\,$ exact complexes
of projective complexes with cardinality less than $\kappa$. Note that we can
identify these quotients such as $\mathbb{P}/\mathbb{Q}^0$ as complexes whose degree $n$
entry is $\oplus_{i \in I_n-L_n} P_{n,i}$ and in doing so we may
continue to find an increasing union $0 \neq \mathbb{Q}^0 \subseteq \mathbb{Q}^1
\subseteq \mathbb{Q}^2 \subseteq \cdots $ corresponding to a nested union of
subsets $L^0_n \subseteq L^1_n \subseteq L^2_n \subseteq \cdots$ for
each $n$. Assuming this process doesn't terminate we set $\mathbb{Q}^{\omega} =
\cup_{\alpha < \omega} \mathbb{Q}^{\alpha}$ and note that $\mathbb{Q}^{\omega}_n =
\oplus_{i \in L^{\omega}_n} P_{n,i}$ where $L^{\omega}_n =
\cup_{\alpha < \omega} L^{\alpha}_n$. So still, $\mathbb{Q}^{\omega}$ and
$\mathbb{P}/\mathbb{Q}^{\omega}$ are complexes of projective complexes and are $A \overline{\otimes} \,-\,$
exact since $A \overline{\otimes} \,-\,$ commutes with direct limits. Therefore we
can continue this process with $\mathbb{P}/\mathbb{Q}^{\omega}$ to obtain $\mathbb{Q}^{\omega
+1}$ with all the properties we desire. Using this process we can
obtain an ordinal $\lambda$ and a continuous union $\mathbb{P} = \cup_{\alpha <
\lambda} \mathbb{Q}^{\alpha}$ with $\mathbb{Q}_{\alpha}, \mathbb{Q}_{\alpha + 1}/\mathbb{Q}_{\alpha}$ all being
$A \overline{\otimes} \,-\,$ exact complexes of projective complexes and having $|\mathbb{Q}_{\alpha}|,
|\mathbb{Q}_{\alpha +1}/\mathbb{Q}_{\alpha}| < \kappa$.
\end{proof}

We can now prove Theorem~\ref{thm-how to create projective on chain}

\begin{proof}
The plan is to apply Proposition~\ref{prop-how to create a projective
model structure}. First let $\kappa > \text{max}\{\, |R| \, , \, |A|
\, , \, \omega \,\}$ be a regular cardinal and let $S$ be the set of
all $A$-acyclic complexes of projective complexes $\mathbb{P} \in \class{C}$ 
such that $|\mathbb{P}| \leq \kappa$. (We
really need to take a representative for each isomorphism class so
that we actually get a set as opposed to a proper class). Now the
set $S$ cogenerates a complete cotorsion pair (by~\cite[Theorem~2.4]{hovey})
$(\leftperp{(\rightperp{S})},\rightperp{S})$ in $\textnormal{Ch}(\ch)$, where the left side consists precisely of all retracts of transfinite extensions of complexes in $S$. But $S \subseteq \class{C}$, and $\class{C}$ is closed under retracts and transfinite extensions, so $\leftperp{(\rightperp{S})} \subseteq \class{C}$. The reverse containment $\class{C} \subseteq \leftperp{(\rightperp{S})}$ comes from Theorem~\ref{theorem-filtrations for complexes of projectives}.
This proves the first part of Proposition~\ref{prop-how to create a projective model structure}.

Setting $\class{W} = \rightperp{\class{C}}$, it is left to show that $\class{W}$ is thick and contains the projective objects of $\textnormal{Ch}(\ch)$. To see that $\cat{W}$ is thick, first note that, because $\cat{C}$ consists of complexes of projective complexes (that is, complexes with projective objects in each degree),
Lemma~\ref{lemma-homcomplex-basic-lemma} implies that $\mathbb{X} \in \cat{W}$
if and only if $\homcomplex (\mathbb{C},\mathbb{X})$ is acyclic for all $\mathbb{C} \in \cat{C}$.  Now
suppose we have a short exact sequence
\[
0 \xrightarrow{} \mathbb{X} \xrightarrow{} \mathbb{Y} \xrightarrow{} \mathbb{Z} \xrightarrow{} 0,
\]
where two out of three of the entries are in $\cat{W}$, and suppose
$\mathbb{C} \in \cat{C}$.  Since each $\mathbb{C}_n$ is a projective object, the resulting
sequence
\[
0 \xrightarrow{} \homcomplex (\mathbb{C},\mathbb{X}) \xrightarrow{} \homcomplex (\mathbb{C},\mathbb{Y}) \xrightarrow{}
\homcomplex (\mathbb{C},\mathbb{Z})\xrightarrow{} 0
\]
is still short exact.  Since two out of three of these complexes are
acyclic, so is the third. This proves thickness of $\class{W}$.

Now if $\mathbb{X}$ is contractible, then $\homcomplex (\mathbb{C},\mathbb{X})$ is obviously
acyclic for any $\mathbb{C}$, so $\mathbb{X} \in \cat{W}$. In particular, $\class{W}$ must contain the projective objects as these are contractible; for example, see~\cite[Lemma~4.5]{gillespie-G-derived}. 

So we have finished proving that the model structure exists. It is
cofibrantly generated because we are working in a Grothendieck category with enough projectives and the cotorsion pair is cogenerated by a set; see the results of~\cite[Section~6]{hovey}.
\end{proof}

\section{The AC-acyclic projective model structure on double complexes}\label{sec-double-complexes}

Let $\class{C}$ be the class of all the complexes of projectives appearing in Definition~\ref{def-Gorenstein AC-projective complex}. That is, $\class{C}$ consists of all exact complexes of projective complexes $$\mathbb{C} \equiv \cdots \rightarrow P_1 \rightarrow P_0 \rightarrow P^0 \rightarrow P^1 \rightarrow \cdots$$ which remain exact after applying $\Hom_{\ch}(-,L)$ for any level chain complex $L$. We show in this brief section that $\class{C}$ is the left half of a projective cotorsion pair, cogenerated by a set in $\textnormal{Ch}(\ch)$.

\begin{lemma}\label{lemma-AC-acyclic-bicomplexes}
Let $\mathbb{C}$ be a complex of projective complexes. The following are equivalent: 
\begin{enumerate}
\item $\mathbb{C} \in \class{C}$. That is, $\mathbb{C}$ remains exact after applying $\Hom_{\ch}(-,L)$ for any level chain complex $L$. 
\item $\mathbb{C}$ remains exact after applying $\overline{\homcomplex}(-,L)$ for any level complex $L$.
\item $\mathbb{C}$ remains exact after applying $A \overline{\otimes} -$ for any absolutely clean chain complex (of right $R$-modules) $A$.
\end{enumerate}
\end{lemma}

\begin{proof}
Referring to Section~\ref{subsec-modified hom and tensor} it is easy to see that the condition $\mathbb{C}$ remains exact after applying $\Hom_{\ch}(-L)$ is equivalent to requiring that it remains exact after applying $\overline{\homcomplex}(-,L)$ for any level complex $L$. But, by Proposition~\ref{prop-dual-exact}, this is equivalent to requiring that it remains exact after applying $A \overline{\otimes} -$ for any absolutely clean chain complex (of right $R$-modules) $A$.
\end{proof}

\begin{lemma}\label{lemma-test-complex}
There exists a single absolutely clean chain complex (of right $R$-modules) $A$ with the property that a complex $\mathbb{C}$ of projective complexes is in the class $\class{C}$ if and only if $A \overline{\otimes} \mathbb{C}$ is exact. 
\end{lemma}

\begin{proof}
We take $A$ to be the direct sum of the disks $D^n(R_R)$ along with all the absolutely clean complexes in a set $\class{S}$ as in Proposition~\ref{prop-level complexes are transfinite extensions}. One can check that $A$ has the desired property. 
\end{proof}

Taking $A$ as in Lemma~\ref{lemma-test-complex} and applying Theorem~\ref{thm-how to create projective on chain} we get the following corollary. We will call a complex $\mathbb{C} \in \class{C}$ an \textbf{AC-acyclic complex of projective complexes}; again, they are the complexes of projectives appearing in Definition~\ref{def-Gorenstein AC-projective complex}.

\begin{corollary}\label{cor-AC-acyclic-model-bicomplexes}
Let $\class{C}$ be the class of all AC-acyclic complexes of projective complexes.  Then there is a cofibrantly generated abelian model structure on double complexes where every object is fibrant, $\class{C}$ is the class of cofibrant objects, and $\class{W} = \rightperp{\class{C}}$ is the class of trivial objects. In other words, $(\class{C}, \rightperp{\class{C}})$ is a projective cotorsion pair in $\textnormal{Ch}(\ch)$.
\end{corollary}

\section{The Gorenstein AC-projective model structure on complexes}\label{sec-Goren-AC-proj}

Our goal now is to prove Theorem~\ref{them-Gorenstein AC-projectives complete cotorsion pair}. So throughout this section we will let $\class{GP}$ denote the class of Gorenstein AC-projective chain complexes, and set $\class{W} = \rightperp{\class{GP}}$. The goal is to show that $(\class{GP}, \class{W})$ is a projective cotorsion pair in $\ch$. The idea is that we just constructed the double complex version of this cotorsion pair in Corollary~\ref{cor-AC-acyclic-model-bicomplexes}, and we use the functor $\mathbb{X} \mapsto \mathbb{X}_0/B_0\mathbb{X}$ to pass the cotorsion pair down to one on $\ch$. Again, this is just a double complex version of the original approach in~\cite{bravo-gillespie-hovey}, though a few simplifications are made in our Lemmas~\ref{lem-proj-cycles-of-W} and~\ref{lemma-retracts}.

\begin{lemma}\label{lem-spheres}
$W \in \class{W}$ if and only
if $S^{n}(W) \in \rightperp{\class{C}}$ for any $n$. In particular, a chain complex $W \in \class{W}$ if and only if it is trivial when viewed as a double complex in the AC-acyclic projective model structure of Corollary~\ref{cor-AC-acyclic-model-bicomplexes}.
\end{lemma}

\begin{proof}
For any abelian category $\cat{A}$, object $W \in \cat{A}$, and exact chain complex $\mathbb{C} \in \cha{A}$, we have an isomorphism $\Ext^{1}_{\cha{A}}(\mathbb{C},S^{n}(W)) \cong \Ext^{1}_{\cat{A}}(\mathbb{C}_n/B_{n}\mathbb{C}, W)$~\cite[Lemma~4.2]{gillespie-degreewise-model-strucs}. Since $\mathbb{C}_n/B_{n}\mathbb{C} \cong Z_{n-1}\mathbb{C}$, the lemma follows immediately from this isomorphism and definitions. 
\end{proof}

\begin{lemma}\label{lem-thick}
$\class{W} = \rightperp{\class{GP}}$ is a thick class and contains all the projective chain complexes. 
\end{lemma}

\begin{proof}
Thickness is immediate from Lemma~\ref{lem-spheres} since $\rightperp{\class{C}}$ is thick. For the projective complexes, note it follows immediately from Definition~\ref{def-Gorenstein AC-projective complex} that $\Ext^n_{\ch}(X,L) = 0$ for any Gorenstein AC-projective chain complex $X$ and level chain complex $L$. In particular, $\Ext^{1}(\mathbb{C},S^{n}(P)) \cong \Ext^{1}_{\ch}(Z_{n-1}\mathbb{C}, P) = 0$ whenever $\mathbb{C}$ is an AC-acyclic complex of projective complexes and $P$ is a projective complex. So $P \in \class{W}$ for any projective complex $P$, by Lemma~\ref{lem-spheres}.
\end{proof}

We need one more lemma concerning the trivial objects. 

\begin{lemma}\label{lem-proj-cycles-of-W}
Suppose $\mathbb{Y}$ is a double complex with $H_{i}\mathbb{Y}=0$ for $i>0$ and $\mathbb{Y}_{i}$ level for $i<0$.  Then $\mathbb{Y}$ is trivial in the AC-acyclic projective model structure of Corollary~\ref{cor-AC-acyclic-model-bicomplexes} if and only if $\mathbb{Y}_{0}/B_{0}\mathbb{Y} \in \class{W}$.
\end{lemma}

\begin{proof}
We first note that any bounded above complex of level complexes is trivial, and any bounded below exact complex of complexes is trivial. Indeed using the definition of an AC-acyclic complex of projective complexes, one verifies that for any level chain complexes $L$, the double complex $S^n(L)$ is trivial in the AC-acyclic projective model structure. That is, $S^n(L) \in \rightperp{\class{C}}$, and so any bounded above complex of level complexes must also be trivial, according to Lemma~\ref{lemma-transfinite extensions of spheres and disks}. On the other hand, one verifies that for any chain complex $X$, the double complex $D^n(X) \in \rightperp{\class{C}}$ too. So Lemma~\ref{lemma-transfinite extensions of spheres and disks} tells us that any bounded below exact complex of complexes is trivial in the AC-acyclic projective model structure.
  
Now the given $\mathbb{Y}$ has a subcomplex $\mathbb{A} \subseteq \mathbb{Y}$, where $\mathbb{A}$ is the shown bounded below exact complex of complexes: $\cdots \xrightarrow{} \mathbb{Y}_2 \xrightarrow{} \mathbb{Y}_1 \xrightarrow{} B_0\mathbb{Y} \xrightarrow{} 0$. As noted above, this complex is trivial, so the given $\mathbb{Y}$ is trivial if and only if the quotient $\mathbb{Y}/\mathbb{A}$ is trivial. We note that this quotient is the complex $0 \xrightarrow{} \mathbb{Y}_{0}/B_{0}\mathbb{Y} \xrightarrow{} \mathbb{Y}_{-1} \xrightarrow{} \mathbb{Y}_{-2} \xrightarrow{} \cdots$, which in turn has another obvious subcomplex  $0 \xrightarrow{} 0 \xrightarrow{} \mathbb{Y}_{-1} \xrightarrow{} \mathbb{Y}_{-2} \xrightarrow{} \cdots$. This is a bounded above complex of level complexes and thus also trivial. So we deduce that $\mathbb{Y}$ is trivial if and only if the corresponding quotient complex, which is $S^0(\mathbb{Y}_{0}/B_{0}\mathbb{Y})$, is trivial. Now looking at Lemma~\ref{lem-spheres} this happens if and only if $\mathbb{Y}_{0}/B_{0}\mathbb{Y} \in \class{W}$. So we have proved the lemma. 
\end{proof}

On the other hand, we will need lemmas concerning the class $\class{GP}$ of Gorenstein AC-projective chain complexes. 

\begin{lemma}\label{lemma-retracts}
Again let $\class{GP}$ denote the class of Gorenstein AC-projective chain complexes.
\begin{enumerate}
\item $\class{GP}$ is closed under direct sums.
\item $\class{GP}$ is projectively resolving in the sense of~\cite[Definition~1.1]{holm}.
\item $\class{GP}$ is closed under retracts (direct summands).
\end{enumerate}
\end{lemma}

\begin{proof}
It is easy to prove (1) straight from Definition~\ref{def-Gorenstein AC-projective complex}.

For (2), let us first recall~\cite[Definition~1.1]{holm}. A class of $R$-modules, or chain complex of $R$-modules, such as $\class{GP}$, is called \emph{projectively resolving} if it contains the projectives and if for any short exact sequence $0 \xrightarrow{} X'  \xrightarrow{} X \xrightarrow{} X'' \xrightarrow{} 0$ with $X'' \in \class{GP}$, the conditions $X' \in \class{GP}$ and $X \in \class{GP}$ are equivalent. The class of all Gorenstein AC-projective $R$-modules was shown to be projectively resolving in~\cite[Section~8]{bravo-gillespie-hovey}. To extend this to the class $\class{GP}$, of Gorenstein AC-projective chain complexes, we use the characterization of Gorenstein AC-projective complexes from~\cite[Theorem~4.13]{bravo-gillespie}: A chain complex $X$ is Gorenstein AC-projective if and only if each $X_n$ is a Gorenstein AC-projective $R$-module and the external Hom, $\homcomplex(X,L)$, is exact whenever $L$ is a level complex. Now given a level complex $L$, we apply $\homcomplex(-,L)$ to the above short exact sequence. We note that $\homcomplex(-,L)$ certainly takes all right exact sequences to left exact sequences and it in fact preserves short exact sequences for which $X''$ is level. Indeed referring to Section~\ref{subsec-complexes} we see that in each degree $n$, we have the exact sequence
$$ \prod_{k \in \Z} \Hom_R(X_{k},L_{k+n}) \xrightarrow{}
\prod_{k \in \Z} \Hom_R(X'_{k},L_{k+n}) \xrightarrow{} \prod_{k \in
\Z} \Ext^1_R(X''_{k},L_{k+n}) = 0.$$ The last product is 0 because we have $\Ext^1_R(M,N) =0$ whenever $M$ is a Gorenstein AC-projective $R$-module and $N$ is a level $R$-module. So now 
$$0 \xrightarrow{} \homcomplex(X'',L)  \xrightarrow{} \homcomplex(X,L) \xrightarrow{} \homcomplex(X',L) \xrightarrow{} 0$$ is a short exact sequence with $\homcomplex(X'',L)$ an exact complex. So $\homcomplex(X',L)$ is exact if and only if $\homcomplex(X,L)$ is exact. We proved (2).

Finally, Holm shows in~\cite[Proposition~1.4]{holm} that an Eilenberg swindle argument can be used to conclude (3) from both (1) and (2). It is clear that the argument given there, for $R$-modules, holds for classes of chain complexes as well. 
\end{proof}

We can now prove the main Theorem~\ref{them-Gorenstein AC-projectives complete cotorsion pair} stated in the Introduction.

\begin{proof}[Proof of Theorem~\ref{them-Gorenstein AC-projectives complete cotorsion pair}]
Again, $\class{GP}$ denotes the class of all Gorenstein AC-projective chain complexes,
and $\class{W}=\rightperp{\class{GP}}$. By Lemma~\ref{lem-thick} we know that $\class{W}$ is thick and contains all projective chain complexes. So by Proposition~\ref{prop-how to create a projective model structure} we will have a projective cotorsion pair once we show that $(\class{GP} ,\class{W})$ is a complete cotorsion pair. Before showing $(\class{GP} ,\class{W})$ is a cotorsion pair we first will show that for a given chain complex $X$, we can find a short exact sequence $0 \xrightarrow{} W \xrightarrow{} P \xrightarrow{} X \xrightarrow{} 0$ with $P$ Gorenstein AC-projective and $W \in \class{W}$. Indeed letting $S^0(X)$ be the double complex with $X$ concentrated in degree 0, we can use the complete cotorsion pair $(\class{C},\rightperp{\class{C}})$ of Corollary~\ref{cor-AC-acyclic-model-bicomplexes} to first obtain a short exact sequence of double complexes
\[
0 \xrightarrow{} \mathbb{Y} \xrightarrow{} \mathbb{C} \xrightarrow{}S^0(X) \xrightarrow{} 0
\]
with $\mathbb{C}$ an AC-acyclic complex of projective complexes and $\mathbb{Y} \in \rightperp{\class{C}}$; so $\mathbb{Y}$ is trivial in the AC-acyclic projective model structure.  
By the snake lemma, we get a short exact sequence
\[
0 \xrightarrow{} \mathbb{Y}_0/B_0\mathbb{Y}  \xrightarrow{} \mathbb{C}_0/B_0\mathbb{C} \xrightarrow{} X \xrightarrow{} 0.
\]
Of course $ \mathbb{C}_0/B_0\mathbb{C} \cong Z_{-1}\mathbb{C}$ is Gorenstein AC-projective by definition, but also $ \mathbb{Y}_0/B_0\mathbb{Y}$
is in $\class{W}$ by Lemma~\ref{lem-proj-cycles-of-W}, since
$\mathbb{Y}_{i}$ is projective (so level) for all $i \neq 0$ and $H_{i}\mathbb{Y}=0$ for all $i\neq
-1$. 

So we have shown that for any chain complex $X$ we can find a short exact sequence $0 \xrightarrow{} W \xrightarrow{} P \xrightarrow{} X \xrightarrow{} 0$ with $P \in \class{GP}$ and $W \in \class{W}$. From this and the fact that $\class{GP}$ is closed under retracts (Lemma~\ref{lemma-retracts}), a standard argument will show that $(\class{GP} ,\class{W})$ is indeed a cotorsion pair, and of course it has enough projectives. But then the so-called ``Salce-trick'' applies and tells us that the cotorsion pair also has enough injectives, and so it is a complete cotorsion pair.

The cotorsion pair $(\class{GP} ,\class{W})$ is cogenerated by the set of all Gorenstein AC-projective complexes with cardinality less than $\kappa$, where $\kappa$ is chosen as in Theorem~\ref{theorem-filtrations for complexes of projectives} (with $A$ as in Lemma~\ref{lemma-test-complex}). Indeed given any Gorenstein AC-projective complex $X$, we have $X = Z_0\mathbb{C}$ for some AC-acyclic complex of projective complexes $\mathbb{C}$. Theorem~\ref{theorem-filtrations for complexes of projectives} shows that $\mathbb{C}$ has a filtration $\mathbb{C} = \cup_{\alpha < \lambda}
\mathbb{Q}_{\alpha}$ where each $\mathbb{Q}_{\alpha}, \mathbb{Q}_{\alpha + 1}/\mathbb{Q}_{\alpha}$ are
also AC-acyclic complexes of projective complexes and $|\mathbb{Q}_{\alpha}|,
|\mathbb{Q}_{\alpha + 1}/\mathbb{Q}_{\alpha}| < \kappa$. It follows that $X = \cup_{\alpha < \lambda}
Z_0\mathbb{Q}_{\alpha}$ is also a filtration of $X$ by the Gorenstein AC-projective complexes $Z_0\mathbb{Q}_{\alpha}$ (with $\kappa$-bounded cardinality).
\end{proof}

The following corollary describes the homotopy category of the Gorenstein AC-projective model structure. It follows from~\cite[Lemma~5.1]{gillespie-hereditary-abelian-models}.

\begin{corollary}\label{cor-homotopy-Gorenstein-AC-proj}
For any ring $R$, the homotopy category of the Gorenstein AC-projective
model structure on $\ch$ is equivalent to the category of all Gorenstein
AC-projective complexes modulo the usual chain homotopy relation.
\end{corollary}

We now relate the main theorem to the existence of certain precovers in $\ch$ that are of interest. First, by referring to Definition~\ref{def-Gorenstein AC-projective complex}, we note that by loosening the requirement ``for any level complex $L$'' to only requiring ``for any flat complex $F$'' (resp. ``for any projective complex $P$'') we reproduce the definition of the \emph{Ding projective} complexes of~\cite{Ding-Chen-complex-models} (resp. \emph{Gorenstein projective} complexes of~\cite{garcia-rozas}). 

\begin{corollary}\label{cor-precovers}
We have the following statements concerning existence of Gorenstein AC-projective, Ding projective, and Gorenstein projective precovers in $\ch$.
\begin{enumerate}
\item Every chain complex over any ring has a special Gorenstein AC-projective precover. 
\item If $R$ is a (right) coherent ring, then every chain complex has a special Ding projective precover. 
\item lf $R$ is any ring in which all level modules have finite projective dimension, then every chain complex has a special Gorenstein projective precover.
\end{enumerate} 
\end{corollary}

ln particular, (3) says that if $R$ is a (right) coherent ring in which all flat (left) modules have finite projective dimension (called \emph{left $n$-perfect}), then every chain complex has a special Gorenstein projective precover. This was also recently established in~\cite{estrada-iacob-odabasi-precovers} and~\cite{gillespie-recollement}. The same results of Corollary~\ref{cor-precovers}, but for $R$-modules, are proved in~\cite{bravo-gillespie-hovey}. 

\begin{proof}
The first statement is clear from the Definition given in Section~\ref{sec-cot}. For the second statement,
if $R$ is a (right) coherent ring, then a chain complex of (left) $R$-modules is level if and only if it is flat~\cite{bravo-gillespie}. So in this case Gorenstein AC-projective coincides with the notion of Ding projective.

For the last statement, suppose all level modules have finite projective dimension. Since level modules are closed under direct sums there must be an upper bound on the projective dimensions. Using the characterization of level complexes from Proposition~\ref{prop-level chain complexes} one can argue that all level complexes also have finite projective dimension (and with the same upper bound on their dimensions). So if $L$ is a level complex then we can take a finite projective resolution $$0 \xrightarrow{} Q_n \xrightarrow{} \cdots \xrightarrow{} Q_2 \xrightarrow{} Q_1 \xrightarrow{} Q_0 \xrightarrow{} L \xrightarrow{} 0.$$ Now if we let $\class{P}_{\circ}$ denote an exact complex of projectives as in Definition~\ref{def-Gorenstein AC-projective complex}, we can apply $\Hom_{\ch}(\class{P}_{\circ}, - )$ to the above resolution of $L$ and argue that if $\Hom_{\ch}(\class{P}_{\circ},Q)$ is exact for any projective chain complex $Q$, then $\Hom_{\ch}(\class{P}_{\circ},L)$ is also exact for $L$. So the notion of Gorenstein AC-projective coincides with the usual notion of Gorenstein projective in this case. 
\end{proof}

In fact, most rings encountered in practice are (one-sided) Noetherian or at least (one-sided) coherent. And we refer the reader to~\cite[Page~892]{gillespie-recollement} for a lengthy discussion of the many rings satisfying the property that every flat module has finite projective dimension. So for most rings encountered in practice the three notions appearing in Corollary~\ref{cor-precovers} coincide. This is also true for the Ding-Chen rings considered in the next section, though over such rings a flat module need not have finite projective dimension; see the Remark at the end of Section~\ref{sec-Ding-Chen}.

\section{The case of Ding-Chen rings}\label{sec-Ding-Chen}

The model structure we just constructed in Section~\ref{sec-Goren-AC-proj} is a cofibrantly generated, hereditary, abelian model structure. As such it is known that its homotopy category is a well-generated triangulated category in the sense of~\cite{neeman-well generated}. We now show it is in fact a compactly generated category in the case that $R$ is a Ding-Chen ring in the sense of~\cite{gillespie-ding}. Such a ring is, by definition, a two-sided coherent ring in which ${}_RR$ and $R_R$ each have finite absolutely pure (FP-injective) dimension. The two-sided Noetherian Ding-Chen rings are precisely the Gorenstein rings of Iwanaga~\cite{iwanaga,iwanaga2}. The main result here is Theorem~\ref{them-Ding-Chen case}. The compactly generated part of the theorem may be viewed as a chain complex analog to a result of Stovicek~\cite[Prop.~7.9]{stovicek-purity}, though our proof is entirely different.

Again, Theorem~\ref{them-Gorenstein AC-projectives complete cotorsion pair} shows that for any ring $R$, we have the projective cotorsion pair $\class{M}_{prj} = (\class{GP},\rightperp{\class{GP}})$, which induces the Gorenstein AC-projective model structure on $\ch$. But we also have the injective cotorsion pair $\class{M}_{inj} = (\leftperp{\class{GI}},\class{GI})$, inducing the Gorenstein AC-injective model structure on $\ch$; see~\cite[Theorem~3.3]{bravo-gillespie}.

\begin{lemma}\label{lemma-Quillen functor}
For any ring $R$, the identity functor is a left Quillen functor from  $\class{M}_{prj} = (\class{GP},\rightperp{\class{GP}})$,
the Gorenstein AC-projective model structure, to $\class{M}_{inj} = (\leftperp{\class{GI}},\class{GI})$, the Gorenstein AC-injective model structure.  
\end{lemma}

\begin{proof}
It is clear that the identity functor takes cofibrations (resp. trivial cofibrations) in the
Gorenstein AC-projective model structure, which are monomorphisms with Gorenstein AC-projective (resp. categorically projective) cokernels, to cofibrations in the Gorenstein AC-injective model
structure, which are monomorphisms with any cokernel (resp. trivial cokernel).  Note that a categorically projective complex $P$ certainly is trivial in the Gorenstein AC-injective model structure because $\Ext^1_{\ch}(P,X) = 0$ for any Gorenstein AC-injective complex $X$. Since the identity functor is a left adjoint (to itself) and preserves cofibrations and trivial cofibrations it is a left Quillen functor by definition. 
\end{proof}

\begin{lemma}\label{lemma-trivials}
Let $R$ be a Ding-Chen ring. That is, a two-sided coherent ring in which ${}_RR$ and $R_R$ each have finite absolutely pure dimension. Then $\rightperp{\class{GP}} = \leftperp{\class{GI}}$. This class, denoted $\class{W}$, consists precisely of all chain complexes having finite flat  (equivalently, absolutely pure) dimension. A chain complex $W$ is in $\class{W}$ if and only if it is exact and each cycle module $Z_nW$ has finite flat (equivalently, absolutely pure) dimension in $R$-Mod. 
\end{lemma}

\begin{proof}
Since $R$ is coherent, a level complex is the same as a flat complex, and so a Gorenstein AC-projective complex is exactly a \emph{Ding projective} complex in the sense of~\cite[Section~3]{Ding-Chen-complex-models}. (Similarly, the Gorenstein AC-injectives coincide with the \emph{Ding injective} complexes.) The result now follows from~\cite[Theorem~4.5]{Ding-Chen-complex-models}.
\end{proof}

\begin{lemma}\label{lemma-AC-complexes}
Let $R$ be a Ding-Chen ring. That is, a two-sided coherent ring in which ${}_RR$ and $R_R$ each have finite absolutely pure dimension. Then the class $\class{GP}$ of Gorenstein AC-projective complexes coincides with the class of (usual) Gorenstein projective complexes, and these are precisely the complexes $X$ having each component $X_n$ a Gorenstein projective $R$-module (in the usual sense of~\cite{enochs-jenda-book}). Similarly, the class $\class{GI}$ of Gorenstein AC-injectives coincides with the class of (usual) Gorenstein injective complexes, and these are precisely the complexes $X$ having each component $X_n$ a Gorenstein injective $R$-module.
\end{lemma}

\begin{proof}
Again since $R$ is coherent, Gorenstein AC-projective coincides with Ding projective and Gorenstein AC-injective coincides with Ding-injective. The result now comes from~\cite[Theorem~1.1/1.2]{gillespie-ding-modules}.  
\end{proof}

We are now ready to prove the main result concerning Gorenstein AC-projectives in the case that $R$ is a Ding-Chen ring.  

\begin{theorem}\label{them-Ding-Chen case}
Let $R$ be a Ding-Chen ring. That is, a two-sided coherent ring in which ${}_RR$ and $R_R$ each have finite absolutely pure dimension. Then the identity functor is a Quillen equivalence from  $\class{M}_{prj} = (\class{GP},\class{W})$,
the Gorenstein AC-projective model structure, to $\class{M}_{inj} = (\class{W},\class{GI})$, the Gorenstein AC-injective model structure.  The associated homotopy category is compactly generated and equivalent to the chain homotopy category of all chain complexes $X$ having each component $X_n$ a Gorenstein projective $R$-module (in the usual sense of~\cite{enochs-jenda-book}). This in turn is equivalent to the chain homotopy category of all chain complexes $X$ having each component $X_n$ a Gorenstein injective $R$-module.
\end{theorem}

\begin{proof}
Lemma~\ref{lemma-Quillen functor} tell us the identity is a left Quillen functor between the two model structures.  Lemma~\ref{lemma-trivials} tells us that the class of trivial objects in the two model structures are equal. It follows that the two homotopy categories are equal and the identity functor becomes a Quillen equivalence in this case~\cite[Lemma~5.4]{estrada-gill-coherent}. Lemma~\ref{lemma-AC-complexes}, along with Corollary~\ref{cor-homotopy-Gorenstein-AC-proj}, (resp.~\cite[Lemma~5.1]{gillespie-hereditary-abelian-models} in the injective case), give us the description of the homotopy category as the chain homotopy category of all complexes $X$ having each $X_n$ a Gorenstein projective (resp. Gorenstein injective) $R$-module. 

It is left to show that we have a compactly generated homotopy category. For this, suppose the dimension of the Ding-Chen ring $R$ is $d$. Let $\class{S} = \{\,\Omega^dF\,\}$ be a set of $d$th syzygies on a set (of isomorphism representatives) of all finitely presented chain complexes $F$. Then, arguing similarly to the proof of~\cite[Theorem~8.3]{hovey},  we can argue that $X \in \rightperp{\class{S}}$ if and only if $X$ has FP-injective (absolutely pure) dimension $\leq d$. Referring to Lemma~\ref{lemma-trivials} this means $\rightperp{\class{S}} = \class{W}$, and so $\class{S}$ cogenerates the cotorsion pair in this case. Note that since $R$ is coherent, the class of finitely presented complexes is closed under taking kernels. So each $\Omega^dF$ can be taken to be finitely presented (f.g. projective complexes are automatically finitely presented.)  Now as in the proof of~\cite[Theorem~9.4]{hovey}, we get from a general theorem~\cite[Corollary~7.4.4]{hovey-model-categories} that the set 
$$I = \{\, \Omega^{d+1}F \hookrightarrow P_d \,\} \cup \{\,0 \hookrightarrow D^n(R)\,\},$$ where $0 \xrightarrow{} \Omega^{d+1}F \xrightarrow{} P_d \xrightarrow{} \Omega^{d}F \xrightarrow{} 0$ is a short exact sequence taken with $P_d$ a finitely generated projective, provides a set of (finite) generating cofibrations. 
$J = \{\,0 \hookrightarrow D^n(R)\,\}$ is the set of (finite) generating trivial cofibrations. So the model structure is finitely generated and hence its homotopy category is compactly generated. 
\end{proof}

\begin{remark}
We continue the remarks made at the end of Section~\ref{sec-Goren-AC-proj}. For the Ding-Chen rings considered in this section, we again have Gorenstein AC-projective = Ding projective = Gorenstein projective. But we note that a flat module over a Ding-Chen ring may not have finite projective dimension. Indeed any von Neumann regular ring is Ding-Chen and such a ring may have infinite global dimension. A particular example is obtained by using the free Boolean rings of~\cite[Section~5]{pierce}. A \emph{Boolean ring} is a ring satisfying the identiy $x^2 = x$; such a ring is commutative and von Neumann regular. Let $F_{\alpha}$ be the free Boolean ring on $\aleph_{\alpha}$ generators. Pierce computes its global dimension in~\cite[Cor.~5.2]{pierce}; it is $\textnormal{dim}(F_{\alpha}) = n+1$ if $\alpha = n < \omega$, and  $\textnormal{dim}(F_{\alpha}) = \infty$ if $\alpha$ is infinite. 
\end{remark}

\providecommand{\bysame}{\leavevmode\hbox to3em{\hrulefill}\thinspace}
\providecommand{\MR}{\relax\ifhmode\unskip\space\fi MR }
\providecommand{\MRhref}[2]{%
  \href{http://www.ams.org/mathscinet-getitem?mr=#1}{#2}
}
\providecommand{\href}[2]{#2}

\end{document}